\documentclass[12pt]{article}

\usepackage{amsmath,amsthm,amssymb,mathrsfs}
\usepackage{graphicx,mdframed}

\font\tencmmib=cmmib10 \skewchar\tencmmib '60
\newfam\cmmibfam
\textfont\cmmibfam=\tencmmib

\def\lessim{\ \lower4pt\hbox{$
\buildrel{\displaystyle <}\over\sim$}\ }
\def\gessim{\ \lower4pt\hbox{$\buildrel{\displaystyle >}
\over\sim$}\ }

\def\si{{\sigma}}

\newcommand{\la}{\langle}
\newcommand{\ra}{\rangle}

\newcommand{\e}{\mathbb{E}}

\newtheorem{lemma}{\bf Lemma}

\newtheorem{theorem}{\bf Theorem}

\newtheorem{remark}{\bf Remark}
\newtheorem{example}{\bf Example}

\newmdtheoremenv{theo}{Theorem}

\newenvironment{Proof of lemma}{\noindent{\bf Proof of Lemma}}{\hfill$\Box$\newline}
\newenvironment{Proof of theorem}{\noindent{\it Proof of Theorem}}{\hfill\scriptsize{$\Box$}\newline}
\newenvironment{Proof of theorems}{\noindent{\bf Proof of Theorems}}{\hfill$\Box$\newline}
\newenvironment{Proof of proposition}{\noindent{\bf Proof of Proposition}}{\hfill$\Box$\newline}
\newenvironment{Proof of propositions}{\noindent{\bf Proof of Propositions}}{\hfill$\Box$\newline}
\newenvironment{Proof of exercise}{\noindent{\it Proof of Exercise:}}{\hfill$\Box$}

\begin{document}

\title{Universality of chaos and ultrametricity in mixed $p$-spin models}

\author{
Antonio Auffinger\thanks{Department of Mathematics, Northwestern University. Email: auffing@math.northwestern.edu}  \\ \small{Northwestern University} \and
Wei-Kuo Chen\thanks{Department of Mathematics, University of Chicago. Email: wkchen@math.uchicago.edu}  \\ \small{University of Chicago}
}
\maketitle

\begin{abstract}
We prove disorder universality of chaos phenomena and ultrametricity in the mixed $p$-spin model under mild  moment assumptions on the environment. This establishes the long-standing belief among physicists that the Parisi solution in mean-field models is universal. Our results extend to universal properties of other physical observables in the mixed $p$-spin model as well as in different spin glass models. These include universality of quenched disorder chaos in the Edwards-Anderson (EA) model and quenched concentration for the magnetization in both EA and mixed $p$-spin models under non-Gaussian environments. In addition, we show quenched self-averaging for the overlap in the random field Ising model under small perturbation of the external field.
\end{abstract}

\section{Introduction}

It is widely expected that many statistical quantities and properties in disordered systems should not depend on the particular distribution of the environment. This phenomenom, described as universality, is a major topic of research within the probability and mathematical physics communities. The simplest example is the central limit theorem, where the limiting distribution of sum of independent and identically distributed random variables is Gaussian provided that these variables have  just a finite second moment. 

In spin glasses, universality is broadly accepted. Most of the fascinating predictions made by physicists should hold for any source of the environment under some moment conditions.
Among these predictions, the Parisi solution for mean field spin glass models stands as one of the most ingenious and important ideas of the past decades. Two major steps in the Parisi solution are known rigorously. First, the Parisi formula, which describes the limiting free energy \cite{Guerra03, P12, Tal06}, and, more recently, the proof of the ultrametricity conjecture by Panchenko \cite{Panchenko}.  In his celebrated work, Panchenko established the ultrametricity conjecture upon the validity of the Ghirlanda-Guerra (GG) identities. These identities are known in several mean-field spin glass models with Gaussian disorder, including the Sherrington-Kirkpatrick (SK) model and the mixed $p$-spin models. They were first proved in \cite{GG} on average over the temperature parameters and later, in a similar formulation, by introducing a perturbation term to the Hamiltonian. For generic mixed spin-models they are known in a strong sense without perturbation \cite{Pan2010}. 

The first main goal of this paper is to remove the hypothesis of Gaussian disorder and to show that these identities are universal. We prove that if a mixed $p$-spin model with Gaussian disorder satisfies the GG identities, then these identities hold in the same model with any disorder that matches the first four moments of a standard normal. This moment assumption is reduced to matching two moments if the model does not contain both $2$-spin and $3$-spin interactions. This result combined with the main theorem of \cite{Panchenko} establishes ultrametricity under these assumptions. In fact, we obtain universality for ultrametricity (and for the GG identities) as a direct consequence of universality for the Gibbs measures of these spin systems (see Theorem 4).

A second main direction in this manuscript is to establish chaos phenomena in a non-Gaussian environment. Chaos is a classical old problem in spin glasses. It arose from the discovery that in some models, a small change in the external parameters, such as the disorder or the temperature will result in a dramatic change to the overall energy landscape. This phenomenon was first predicted by Fisher and Huse \cite{FHuse} for the EA  model, although it seems the term chaos first appeared in the famous paper of Bray and Moore \cite{BM} on disorder chaos for the SK model. It has attracted a lot of recent attention from mathematicians. Chaos in disorder for mixed even-spin models was established in the work of Chatterjee \cite{Chatterjee09} and Chen \cite{Chen11}. Temperature chaos, a more intricate subject according to physicists (see Crisanti and Rizzo \cite{CriRizzo03} for instance),  was obtained by Chen and Panchenko \cite{Chen12} and Chen \cite{Chen13}. Despite the remarkable progress, all these results require Gaussian disorder for both original and perturbed Hamiltonian. The second main goal of this paper is to remove this assumption. We establish both temperature and disorder chaos in the mixed $p$-spin model for environments with the same moment assumptions as in the last paragraph.
%
%To be more precise, a typical way of measuring the instability of a spin system is to sample independently a configuration from a Gibbs measure $G$ and a second configuration from a new Gibbs measure $G'$ constructed as a perturbation of $G$, and consider the behavior of the overlap of these two configurations under $\e \, G \times G'$ . The phenomenon of chaos states that this overlap behaves differently and is concentrated near a constant no matter if the two systems $G$ and $G'$ are in the high or low temperature regime. 
%
%A word of comment is needed here. In a first sight, one may be puzzled by the joint presence of universality and disorder chaos. Roughly speaking, universality means that statistics of configurations sampled independently from the {\it same} random Gibbs measure do not depend on the law of the disorder. On the other hand, if one sample two configurations from {\it different} Gibbs measures then they behave differently. Our results show that regardless of the law of the disorder, if one adds a perturbation to the Hamiltonian then chaos holds. In other words, we show that disorder chaos is also universal.

Our methods also allow us to study concentration and other universal properties in spin systems under the assumption of an i.i.d. environment with mean zero and variance one. In the mixed $p$-spin, we show that the magnetization and the disorder chaos cross overlap are quenched self-averaged. In the EA model and in the random field Ising model we establish self-averaging of the magnetization under a small external field perturbation. Last, in the EA model, we prove quenched self-averaging for the bond overlap in the problem of disorder chaos if the decoupling rate of the environments between the two systems converges to zero as the system size tends to infinity. This  extends the work of Chen and Panchenko \cite{Chen14} where these results were obtained in the Gaussian environment.

Rigorous progress on universality of spin glasses was obtained in the past. It is known that the limiting free energy for the SK does not depend on the particular distribution of the environment. This result was first obtained by Talagrand \cite{Tal02} who showed that Bernoulli and Gaussian disorder share the same limiting free energy and then generalized in two beautiful papers, one by Guerra and Toninelli \cite{GuerraToninelli} under assumption of symmetric laws with finite fourth moment and another by Carmona and Hu \cite{CHu} only assuming mean zero and finite third moment for the environment. The third moment condition was reduced to finite second moment in Chatterjee \cite{Chatt2nd}. It was also extended to other types of mean-field models, for instance, the bipartite model by Genovese \cite{Genovese}. To the best of our knowledge, no universality results were obtained in the past at the level of the Gibbs measure.

The method of our proofs is as follows. As in the papers dealing with the limiting free energy, our first step is inspired by Guerra's interpolation technique: we apply an approximate integration by parts lemma to an appropriate observable (that changes depending on the problem). The more moments we assume for the disorder, the better is this approximation. The core of the argument and the main novelty of our proofs is on how we handle and control the error terms coming from these approximations. We do this by estimating appropriate derivatives and although not very enlightening, the proof goes through an unavoidable sequence of careful computations. 

In this regard, we hope to help the reader through the organization of the paper. In the next section, we show how the method works, where the computations are significantly simpler. There, we prove self-averaging of the magnetization in the mixed $p$-spin model, EA and random field Ising model. In Section \ref{sec2}, we raise the stakes and prove universality of quenched disorder chaos.
Our main results, universality of ultrametricity and temperature and disorder chaos, are in Section \ref{sec4}, where universality of the Gibbs measure is established. 

\subsection{Acknowledgments} Both authors would like to thank the 2014 Midwest Probability Colloquium where part of this work was discussed. W.-K. C. thank D. Panchenko for previous conversations.  The research of A. A. is supported by NSF Grant DMS-1407554.

\section{Quenched self-averaging of the magnetization}\label{subsec1}

Let $\Sigma$ be a finite set and $\mu$ be a random measure on $\Sigma.$ For a given countable set $E$, consider a family of measurable functions $(f_e)_{e\in E}$ with $|f_e|\leq 1$ on $\Sigma$. Let $(y_e)_{e\in E}$ be independent random variables with mean zero and variance one. These are all independent of $\mu.$ We consider the Hamiltonian for $\gamma\geq 0,$
\begin{align}
\label{eq7}
H_y(\sigma)=\gamma\sum_{e\in E}y_ef_e(\sigma)
\end{align}
and define its Gibbs measure as
\begin{align}\label{eq8}
G_y(\sigma)=\frac{\exp H_y(\sigma)d\mu(\sigma)}{Z_y},
\end{align}
where $Z_y$ is the normalizing factor. Denote by the Gibbs average by $\la \cdot\ra_y$ and define the magnetization by $$
M(\sigma)=\frac{1}{|E|}\sum_{e\in E}f_e(\sigma).
$$

\begin{theorem}\label{thm3}
Let $|E|$ be finite and $\gamma>0$. For any $K\geq 1,$
\begin{align}\label{thm3:eq2}
\e\la(M(\sigma)-\la M(\sigma)\ra_y)^2\ra_y&\leq I(K)+\frac{1}{\gamma\sqrt{|E|}},
\end{align}
where
\begin{align*}
I(K):=\frac{8}{|E|}\sum_{e\in E}\e[|y_e|^2;|y_e|\geq K]+9\gamma K.
\end{align*}
In addition, if $$B_3:=\max_{e\in E}\e|y_e|^3 < \infty,$$ then
\begin{align}
\label{thm3:eq1}
\e\la(M(\sigma)-\la M(\sigma)\ra_y)^2\ra_y&\leq 9B_3\gamma+\frac{1}{\gamma\sqrt{|E|}}.
\end{align}
\end{theorem}

\begin{lemma}[Approximate integration by parts]\label{lem2}
Let $y$ be a random variable such that its first $k\geq 2$ moments match those of a Gaussian random variable. Suppose that $F\in C^{k+1}(\mathbb{R})$. 
For any $K\geq 1,$
\begin{align}
\begin{split}
\label{aip:eq1}
|\e yF(y)-\e F'(y)|&\leq\frac{4\|F^{(k-1)}\|_\infty}{(k-2)!}\e[|y|^{k};|y|\geq K]+\frac{(k+1)K}{k!}\|F^{(k)}\|_\infty \e |y|^k.
\end{split}
\end{align}
In addition, if $\e|y|^{k+1}<\infty,$ we have 
\begin{align}\label{aip:eq2}
|\e yF(y)-\e F'(y)|\leq\frac{(k+1)}{k!} \|F^{(k)}\|_\infty \e|y|^{k+1}.
\end{align}
\end{lemma}

\begin{proof}[Proof of Theorem \ref{thm3}] A direct computation gives
\begin{align*}
\frac{d}{dy_e}\la M(\sigma)\ra&=\gamma(\la f_e(\sigma)M(\sigma)\ra_y-\la f_e(\sigma)\ra\la M(\sigma)\ra_y),\\
\frac{d^2}{dy_{e}^2}\la M(\sigma)\ra&=\gamma^2 (\la f_e(\sigma)^2 M(\sigma)\ra_y-\la f_e(\sigma)M(\sigma)\ra_y\la f_e(\sigma)\ra_y)\\
&-\gamma^2(\la f_e(\sigma^1)M(\sigma^2)(f_e(\sigma^1)+f_e(\sigma^2))\ra_y-2\la f_e(\sigma^1)M(\sigma^2)\ra_y\la f_e(\sigma) \ra_y).
\end{align*}
and thus, using approximate integration by parts \eqref{aip:eq1} with $k=2$, we obtain that for any $K\geq 1,$
\begin{align*}
&|\e y_e\la M(\sigma)\ra_y-\gamma\e(\la f_e(\sigma)M(\sigma)\ra_y-\la f_e(\sigma)\ra_y\la M(\sigma)\ra_y) |\leq I_e(K),
\end{align*}
where 
\begin{align*}
I_e(K)&:=8\gamma \e[|y_e|^2;|y_e|\geq K]+9\gamma^2K.
\end{align*}
Dividing this inequality by $\gamma|E|$ and summing over all $e\in E$, the triangle inequality gives
\begin{align*}
\e \la (M(\sigma)-\la M(\sigma)\ra_y)^2\ra_y&\leq \frac{1}{\gamma|E|}\sum_{e\in E}I_e(K)+\frac{1}{\gamma|E|}\e\Bigl|\sum_{e\in E}y_e\Bigr||\la M(\sigma)\ra_y|\\
&\leq I(K)+\frac{1}{\gamma|E|}\e\Bigl|\sum_{e\in E}y_e\Bigr|\\
&\leq I(K)+\frac{1}{\gamma \sqrt{|E|}},
\end{align*}
where in the last inequality, we used Jensen's inequality.
Similarly, if $B_3<\infty$, using \eqref{aip:eq2} with $k=2$ gives
\begin{align*}
&|\e y_e\la M(\sigma)\ra_y-\gamma\e(\la f_e(\sigma)M(\sigma)\ra_y-\la f_e(\sigma)\ra_y\la M(\sigma)\ra_y) |\leq 9\gamma^2\e|y_e|^3
\end{align*}
and therefore,
\begin{align*}
\e \la (M(\sigma)-\la M(\sigma)\ra_y)^2\ra_y&\leq \frac{9\gamma}{|E|}\sum_{e\in E}\e|y_e|^3+\frac{1}{\gamma|E|}\e\Bigl|\sum_{e\in E}y_e\Bigr||\la M(\sigma)\ra_y|\\
&\leq 9\gamma B_3+\frac{1}{\gamma|E|}\e\Bigl|\sum_{e\in E}y_e\Bigr|\\
&\leq 9\gamma B_3+\frac{1}{\gamma\sqrt{|E|}}.
\end{align*}
This completes our proof.
\end{proof}

Now we illustrate the applications of Theorem \ref{thm3} by considering three models: the mixed $p$-spin model, the EA model and the random field Ising model.

\begin{example}[mixed $p$-spin model] \rm Let $(\beta_p)_{p\geq 2}$ be a sequence of nonnegative numbers with $\sum_{p\geq 2}2^p\beta_p^2<\infty$ and $h\in\mathbb{R}$. For any $N\geq 1$, the mixed $p$-spin model is defined as 
$$
X(\sigma)=\sum_{p\geq 2}\beta_pX_{p}(\sigma)+h\sum_{i=1}^N\sigma_i,
$$
for all $\sigma=(\sigma_1,\ldots,\sigma_N)\in\Sigma_N:=\{-1,1\}^N,$ where $X_{p}$ is the pure $p$-spin Hamiltonian, $$
X_{p}(\sigma)=\frac{1}{N^{(p-1)/2}}\sum_{1\leq i_1,\ldots,i_p\leq N}y_{i_1,\ldots,i_p}\sigma_{i_1}\cdots\sigma_{i_p}.
$$
Here $y_{i_1,\ldots,i_p}$'s are independent random variables with mean zero and variance one. We denote the Gibbs measure and the Gibbs expectation of the mixed $p$-spin model as $G$ and $\la\cdot\ra.$ Suppose that $\beta_p>0$ for some $p\geq 2$. Set $E=\{1,\ldots,N\}^p$ and $\gamma=\beta_pN^{-(p-1)/2}.$ Define $y_e=y_{i_1,\ldots,i_p}$ and $f_e(\sigma)=\sigma_{i_1}\cdots\sigma_{i_p}$ if $e=(i_1,\ldots,i_p)\in E.$ Then using the notations from above, $G$ can be rewritten as
\begin{align*}
G(\sigma)&=\frac{\exp H_y(\sigma)d\mu(\sigma)}{Z_y},
\end{align*}
where $$
d\mu(\sigma)=\exp(X(\sigma)-H_y(\sigma)).
$$
Let $m(\sigma)=\sum_{i=1}^N\sigma_i/N.$ Assume only mean zero and variance one for $(y_e)_{e\in E}$, we have that for any $\varepsilon>0$, letting $K=\varepsilon\gamma^{-1}$ and using \eqref{thm3:eq2} lead to
\begin{align*}
\e \la (m(\sigma)^p-\la m(\sigma)^p\ra)^2\ra
&\leq \frac{8\sum_{e\in E}\e[|y_e|^2;|y_e|\geq \varepsilon\gamma^{-1}]}{|E|}+9\varepsilon +\frac{1}{\beta_p\sqrt{N}},
\end{align*}
which gives $\lim_{N\rightarrow\infty}\e \la (m(\sigma)^p-\la m(\sigma)^p\ra)^2\ra=0$ provided the first term on the right-hand side converges to zero as $N$ tends to infinity for any $\varepsilon>0$. This includes for instance the case when $(y_e)_{e\in E}$ are identically distributed. If we have finite third moment assumption $B_3<\infty,$ then we can further obtain the rate of convergence from \eqref{thm3:eq1},
\begin{align*}
\e \la (m(\sigma)^p-\la m(\sigma)^p\ra)^2\ra&\leq \frac{9\beta_pB_3}{N^{(p-1)/2}}+\frac{1}{\beta_p\sqrt{N}}\leq \frac{1}{\sqrt{N}}\Bigl(9\beta_pB_3+\frac{1}{\beta_p}\Bigr).
\end{align*} 
This inequality extends the result obtained in Gaussian disorder in \cite[Example 5]{Chen14}, where the upper bound is $1/\beta_p\sqrt{N}.$

\end{example}

\begin{example}[EA model] \rm Let $A$ be a finite undirected graph with vertex set $v(A)$ and edge set $e(A)$. The EA model with temperature $\beta$ and external field $\gamma$ is defined as
$$
X(\sigma)=\beta\sum_{(i,j)\in e(A)}y_{i,j}\sigma_i\sigma_j+\gamma\sum_{i\in v(A)}y_i\sigma_i
$$
for $\sigma\in\{-1,+1\}^{|v(A)|},$ where $(y_{i,j})_{(i,j)\in e(A)}$ and $(y_i)_{i\in v(A)}$ are independent random variables with mean zero and variance one. Let $G$ and $\la\cdot\ra$ be the Gibbs measure and Gibbs expectation associated to $X$. Set $E=v(A)$ and also $y_e=y_i$ and $f_e(\sigma)=\sigma_i$ if $e=i.$ Then we can rewrite
\begin{align*}
G(\sigma)&=\frac{\exp H_y(\sigma)\mu(\sigma)}{Z_y},
\end{align*} 
where $$
\mu(\sigma)=\exp (X(\sigma)-H_y(\sigma)).
$$
Let $m(\sigma):=|v(A)|^{-1}\sum_{i\in v(A)}\sigma_i.$ For any $\varepsilon>0,$ applying \eqref{thm3:eq2} with $K=\varepsilon\gamma^{-1}$ implies
\begin{align}\label{ex2:eq1}
\e\la(m(\sigma)-\la m(\sigma)\ra)^2 \ra\leq \frac{8\sum_{e\in E}\e[|y_e|^2;|y_e|\geq \varepsilon\gamma^{-1}]}{|E|}+9\varepsilon+\frac{1}{\gamma\sqrt{|v(A)|}}.
\end{align}
If we assume that the external field is a small perturbation with the following decay rate,
\begin{align}
\label{ex2:eq2}
\lim_{|v(A)|\rightarrow\infty}\gamma= 0,\,\,\lim_{|v(A)|\rightarrow\infty} \gamma\sqrt{|v(A)|}=\infty,\,\,\lim_{|v(A)|\rightarrow\infty}\frac{\sum_{e\in E}\e[|y_e|^2;|y_e|\geq \varepsilon\gamma^{-1}]}{|E|}=0
\end{align}
for any $\varepsilon>0,$ then the left-hand side of \eqref{ex2:eq1} converges to zero and thus, the magnetization is concentrated near its Gibbs average.

\end{example}

\begin{example}[Random field Ising model] \rm Let $A$ be the undirected graph with vertex set $v(A)=\mathbb{Z}^d\cap[0,N]^d$ and edge set $e(A)=\{(i,j):i,j\in v(A)\,\,\mbox{with}\,\,|i-j|=1\}.$ Let $\beta,\gamma>0.$ The Hamiltonian of the random field Ising model on $A$ is defined as
$$
X(\sigma)=\beta\sum_{(i,j)\in e(A)}\sigma_i\sigma_j+\gamma\sum_{i\in v(A)}y_i\sigma_i
$$
for $\sigma\in \{-1,+1\}^{v(A)},$ where $(y_i)_{i\in v(A)}$ are independent random variables with mean zero and variance one. As before, we use $\la\cdot\ra$ to denote the Gibbs expectation with respect to the Hamiltonian. One important feature in this model is that its spin correlation satisfies the FKG inequality, 
\begin{align}\label{fkg}
\la \sigma_i\sigma_j\ra\geq \la \sigma_i\ra\la \sigma_j\ra
\end{align}
for all $i,j\in v(A).$ This implies that the site overlap $R_{1,2}=|v(A)|^{-1}\sum_{i\in v(A)}\sigma_i^1\sigma_i^2$ satisfies 
\begin{align*}
\e \la(R_{1,2}-\la R_{1,2}\ra)^2\ra 
&= 
\frac{1}{|v(A)|^2}\sum_{i,j} \e \bigl( \la \sigma_i \sigma_j\ra^2 -\la\sigma_i\ra^2 \la\sigma_j\ra^2 \bigr)
\\
&= \frac{1}{|v(A)|^2}\sum_{i,j} \e \bigl| \la \sigma_i \sigma_j -\la\sigma_i\ra \la\sigma_j\ra \bigr| 
\bigl| \la \sigma_i \sigma_j + \la\sigma_i\ra \la\sigma_j\ra \bigr|
\\
&\leq \frac{2}{|v(A)|^2}\sum_{i,j}\e \bigl| \la\sigma_i\sigma_j\ra-\la\sigma_i\ra\la\sigma_i\ra \bigr|
\\
&=\frac{2}{|v(A)|^2}\sum_{i,j}\e \bigl( \la\sigma_i\sigma_j\ra-\la\sigma_i\ra\la\sigma_j\ra \bigr)\\
&=2\e\la(m(\sigma)-\la m(\sigma)\ra)^2,
\end{align*}
where the third equality used the FKG inequality \eqref{fkg}. Following the same reason as Example 2, if the external field is a small perturbation with decay rate \eqref{ex2:eq2}, then the magnetization is self-averaged, which implies the self-averaging of the overlap by using the inequality above.
\end{example}

\section{Quenched disorder chaos}\label{sec2}

In this section we prove quenched disorder chaos. Recall the notations at the beginning of Section \ref{subsec1}. Let $(y_{1,e})_{e\in E}$ and $(y_{2,e})_{e\in E}$ be independent random variables with mean zero and variance one. We assume that these are also independent of $(y_e)_{e\in E}$. For $0\leq t\leq 1,$ we consider the Hamiltonians on $\Sigma,$
\begin{align}
\begin{split}\label{eq5}
H_{1,y}(\rho)&=\gamma_1\sum_{e\in E}f_e(\rho)(\sqrt{t}y_e+\sqrt{1-t}y_{1,e}),\\
H_{2,y}(\tau)&=\gamma_2\sum_{e\in E}f_e(\tau)(\sqrt{t}y_e+\sqrt{1-t}y_{2,e}).
\end{split}
\end{align}
Note that if $\gamma_1=\gamma_2$ and $t=1,$ then the two systems are identically the same. In the disorder chaos problem, we will be interested in understanding how the coupled system behaves when $H_{1,y}$ and $H_{2,y}$ are slightly decoupled through the parameter $t<1.$ Let $\mu_1,\mu_2$ be two random measures on $\Sigma$. Set the Gibbs measures as
\begin{align}
\begin{split}\label{eq6}
dG_{1,y}(\rho)&=\frac{\exp H_{1,y}(\rho)d\mu_1(\rho)}{Z_{1,y}},\\
dG_{2,y}(\tau)&=\frac{\exp H_{2,y}(\tau)d\mu_2(\tau)}{Z_{2,y}},
\end{split}
\end{align}
where $\gamma_1,\gamma_2\geq 0$ and $Z_{1,y},Z_{2,y}$ are the normalizing factors. Denote by $(\rho^\ell,\tau^\ell)_{\ell\geq 1}$ a sequence of i.i.d. samplings from the product measure $G_{1,y}\times G_{2,y}$. We shall again use the notation $\la\cdot\ra_y$ to stand for the Gibbs expectation but with respect to $\prod_{\ell\geq 1}(G_{1,y}\times G_{2,y}).$ As explained in the introduction, the quantity of great interest is the overlap between the two systems,
\begin{align*}
Q_{\ell,\ell'}&=\frac{1}{|E|}\sum_{e\in E}f_e(\rho^\ell)f_e(\tau^{\ell'}).
\end{align*}
We formulate the main result of this section in a way that will cover the mixed $p$-spin model and the EA models as examples:

\begin{theorem}\label{thm4}
Let $\gamma_1,\gamma_2>0$ and $t\in [0,1)$. For any $K\geq 1,$ we have
\begin{align}\label{thm4:eq1}
\e \la(Q_{1,1}-\la Q_{1,1}\ra_y)^2\ra_y&\leq I(K)+\frac{4(\gamma_1+\gamma_2)}{\gamma_1\gamma_2\sqrt{|E|(1-t)}},
\end{align}
where 
\begin{align*}
I(K)&:=\frac{16}{|E|}\sum_{e\in E}\bigl(\e[|y_{1,e}|^2;|y_{1,e}|\geq K]+\e[|y_{2,e}|^2;|y_{2,e}|\geq K]\bigr)+36K(\gamma_1+\gamma_2)\sqrt{1-t}.
\end{align*}
If $B_{1,3}:=\max_{e\in E}\e |y_{1,e}|^3$ and $B_{2,3}:=\max_{e\in E}\e |y_{2,e}|^3$ are finite, then
\begin{align}\label{thm4:eq2}
\e \la(Q_{1,1}-\la Q_{1,1}\ra_y)^2\ra_y&\leq 36(B_{1,3}\gamma_1+B_{2,3}\gamma_2)\sqrt{1-t}+\frac{4(\gamma_1+\gamma_2)}{\gamma_1\gamma_2\sqrt{|E|(1-t)}}.
\end{align} 
\end{theorem}

\begin{proof}
Note that the first and second derivatives of $\la Q_{1,1}f_e(\rho^1)\ra_y$ with respect to the variable $y_{2,e}$ equal 
\begin{align*}
\frac{d}{dy_{2,e}}\la Q_{1,1}f_e(\rho^1)\ra_y&=\sqrt{1-t}\gamma_2\la Q_{1,1}f_e(\rho^1)(f_e(\tau^1)-f_e(\tau^2))\ra_y,\\
\frac{d^2}{dy_{2,e}^2}\la  Q_{1,1}f_e(\rho^1)\ra_y&=(1-t)\gamma_2^2\la Q_{1,1}f_e(\rho^1)(f_e(\tau^1)^2-f_e(\tau^1)f_e(\tau^2))\ra_y\\
&-(1-t)\gamma_2^2\la Q_{1,1}f_e(\rho^1)f_e(\tau^2)(f_e(\tau^1)+f_e(\tau^2)-f_e(\tau^3)-f_e(\tau^4))\ra_y.
\end{align*}
From the approximate integration by parts \eqref{aip:eq1},
\begin{align*}
|\e \la y_{2,e}Q_{1,1}f_e(\rho^1)\ra_y-\gamma_2\sqrt{1-t}\e\la Q_{1,1}f_e(\rho^1)(f_e(\tau^1)-f_e(\tau^2))\ra_y|&\leq I_{2,e}(K),
\end{align*}
where 
\begin{align*}
I_{2,e}(K)&:=4\gamma_{2}\sqrt{1-t}\cdot\e[|y_{2,e}|^2;|y_{2,e}|\geq K]+9\gamma_2^2(1-t)K.
\end{align*}
Summing over all $e$ and dividing by $|E|\gamma_2 \sqrt{1-t}$, the triangle inequality gives
\begin{align*}
\e\la (Q_{1,1}^2-Q_{1,1}Q_{1,2})\ra_y&\leq \frac{\sum_{e\in E}I_{2,e}(K)}{|E|\gamma_2\sqrt{1-t}}+\frac{1}{\gamma_2|E|\sqrt{1-t}}\e\Bigl\la Q_{1,1}\sum_{e\in E}y_{2,e}f_e(\rho^1)\Bigr\ra_y.
\end{align*}
Observe that $|\e\la Q_{1,1}\sum_{e\in E}y_{2,e}f_e(\rho^1)\ra|\leq \e\la |\sum_{e\in E}y_{2,e}f_e(\rho^1)|\ra.$ Since the Gibbs expectation is only with respect to $\rho^1$ and the disorder in Hamiltonian $H_{1,y}$ is independent of $y_{2,e},$ using conditional expectation and then the Cauchy-Schwarz inequality implies 
\begin{align*}
\Bigl|\e\Bigl\la Q_{1,1}\Bigl(\sum_{e\in E}y_{2,e}f_e(\rho^1)\Bigr)\Bigr\ra_y\Bigr|&\leq\e\Bigl\la \e_2\Bigl|\sum_{e\in E}y_{2,e}f_e(\rho^1)\Bigr|\Bigr\ra\\
&\leq \e\Bigl\la \Bigl(\e_2\Bigl|\sum_{e\in E}y_{2,e}f_e(\rho^1)\Bigr|^2\Bigr)^{1/2}\Bigr\ra\\
&=\e \Bigl\la\Bigl(\sum_{e\in E}f_e(\rho^1)^2\Bigr)^{1/2}\Bigr\ra\leq \sqrt{|E|},
\end{align*}
where $\e_2$ is the expectation with respect to the randomness $(y_{2,e})_{e\in E}.$
Thus, we have
\begin{align*}
\e\la (Q_{1,1}^2-Q_{1,1}Q_{1,2})\ra_y&\leq\frac{\sum_{e\in E}I_{2,e}(K)}{|E|\gamma_2\sqrt{1-t}}+\frac{1}{\gamma_2\sqrt{|E|(1-t)}}
\end{align*}
and similarly,
\begin{align*}
\e\la (Q_{2,2}^2-Q_{2,2}Q_{1,2})\ra_y&\leq \frac{\sum_{e\in E}I_{1,e}(K)}{|E|\gamma_1\sqrt{1-t}}+\frac{1}{\gamma_1\sqrt{|E|(1-t)}}.
\end{align*}
Combining these two inequalities lead to
\begin{align*}
\e \la(Q_{1,1}-\la Q_{1,1}\ra_y)^2\ra_y&\leq \e\la (Q_{1,1}-Q_{2,2})^2\ra_y\\
&\leq 2\e \la(Q_{1,1}-Q_{1,2})^2\ra_y+2\e\la (Q_{1,2}-Q_{2,2})^2\ra_y\\
&= 4\e\la(Q_{1,1}^2-Q_{1,1}Q_{1,2})\ra_y+4\e\la(Q_{2,2}^2-Q_{2,2}Q_{1,2})\ra_y\\
&\leq I(K)+\frac{4(\gamma_1+\gamma_2)}{\gamma_1\gamma_2\sqrt{|E|(1-t)}},
\end{align*}
which finishes the proof for \eqref{thm4:eq1}. As for \eqref{thm4:eq2}, it can be treated exactly in the same way by using the inequality \eqref{aip:eq2} instead.

\end{proof}

\begin{example}[Quenched disorder chaos in the mixed $p$-spin model]\label{ex:4} \rm Recall the Hamiltonians from $X$ and $X_p$ from Example 1. For each $p\geq 2,$ consider two pure $p$-spin Hamiltonians on $\Sigma_N,$
\begin{align*}
X_{1,p}(\rho)&=\frac{1}{N^{(p-1)/2}}\sum_{1\leq i_1,\ldots,i_p\leq N}y_{1,i_1,\ldots,i_p}\rho_{i_1}\cdots\rho_{i_p},\\
X_{2,p}(\tau)&=\frac{1}{N^{(p-1)/2}}\sum_{1\leq i_1,\ldots,i_p\leq N}y_{2,i_1,\ldots,i_p}\tau_{i_1}\cdots\tau_{i_p}.
\end{align*}
Here the random variables $y_{1,i_1,\ldots,i_p}$ and $y_{2,i_1,\ldots,i_p}$ are independent random variables with mean zero and variance one for all $i_1,\ldots,i_p$ and $p\geq 2.$ These are also independent of $(y_{e})_{e\in E}$. Let $(\beta_{1,p})_{p\geq 2}$ and $(\beta_{2,p})_{p\geq 2}$ be two nonnegative sequences with finite $\sum_{p\geq 2}2^p\beta_{1,p}^2$ and $\sum_{p\geq 2}2^p\beta_{2,p}^2$, $h_1,h_2\in\mathbb{R}$ and $(t_p)_{p\geq 2}$ be a sequence with $0\leq t_p\leq 1$ for $p\geq 2.$ We consider two mixed $p$-spin models,
\begin{align}\label{eq10}
X_1(\rho)&=\sum_{p\geq 2}\beta_{1,p}\bigl(\sqrt{t_p}X_p(\rho)+\sqrt{1-t_p}X_{1,p}(\rho)\bigr)+h_1\sum_{1\leq i\leq N}\rho_i,\\
X_2(\tau)&=\sum_{p\geq 2}\beta_{2,p}\bigl(\sqrt{t_p}X_p(\tau)+\sqrt{1-t_p}X_{2,p}(\tau)\bigr)+h_2\sum_{1\leq i\leq N}\tau_i.
\end{align}
and set their Gibbs measures as
\begin{align*}
G_1(\rho)&=\frac{\exp X_1(\rho)}{Z_1},\\
G_2(\tau)&=\frac{\exp X_2(\tau)}{Z_2},
\end{align*}
where $Z_1,Z_2$ are the partition functions. Let $\la\cdot\ra$ denote the Gibbs expectation with respect to $G_1\times G_2.$ Suppose that $\beta_{1,p},\beta_{2,p}>0$ and $0\leq t_p<1$ for some $p\geq 2.$ In the notations of \eqref{eq5} and \eqref{eq6}, we set $E=\{1,\ldots,N\}^p$, $t=t_p$ and for all $e=(i_1,\ldots,i_p)\in E$,
\begin{align*}
f_e(\sigma)&=\sigma_{i_1}\cdots\sigma_{i_p},\\
y_{1,e}&=y_{1,i_1,\ldots,i_p},\\
y_{2,e}&=y_{2,i_2,\ldots,i_p}.
\end{align*}
We can then rewrite
\begin{align*}
G_1(\rho)&=\frac{\exp H_{1,y}(\rho)\mu_1(\rho)}{Z_{1,y}},\\
G_2(\tau)&=\frac{\exp H_{2,y}(\tau)\mu_2(\tau)}{Z_{2,y}},
\end{align*}
where $\gamma_1=\beta_{1,p}N^{-(p-1)/2}$, $\gamma_2=\beta_{2,p}N^{-(p-1)/2}$ and
\begin{align*}
\mu_1(\rho)&=\exp (X_1(\rho)-H_{1,y}(\rho)),\\
\mu_2(\tau)&=\exp (X_2(\tau)-H_{2,y}(\tau)).
\end{align*} 
Set the overlap between the two systems as
$$
R_{1,1}=\frac{1}{N}\sum_{i=1}^N\rho_i^1\tau_i^1.
$$ 
For any $\varepsilon>0,$ applying \eqref{thm4:eq1} with $K=\varepsilon N^{(p-1)/2}$ gives
\begin{align*}
\e \la (R_{1,1}^p-\la R_{1,1}^p\ra)^2\ra&\leq\frac{16}{|E|}\sum_{e\in E}\bigl(\e[|y_{1,e}|^2;|y_{1,e}|\geq \varepsilon N^{(p-1)/2}]+\e[|y_{2,e}|^2;|y_{2,e}|\geq \varepsilon N^{(p-1)/2}]\bigr)\\
&+36\varepsilon(\beta_{1,p}+\beta_{2,p})\sqrt{1-t_p}+\frac{4(\beta_{1,p}+\beta_{2,p})}{\beta_{1,p}\beta_{2,p}\sqrt{N(1-t_p)}}.
\end{align*}
Therefore, if 
\begin{align*}
\lim_{N\rightarrow\infty}\frac{1}{|E|}\sum_{e\in E}\bigl(\e[|y_{1,e}|^2;|y_{1,e}|\geq \varepsilon N^{(p-1)/2}]+\e[|y_{2,e}|^2;|y_{2,e}|\geq \varepsilon N^{(p-1)/2}]\bigr)=0
\end{align*}
for any $\varepsilon>0,$ then $R_{1,1}^p$ is concentrated near its Gibbs average, which amounts to say that $R_{1,1}\approx \la R_{1,1}\ra$ if $p$ is odd and $|R_{1,1}|\approx\la |R_{1,1}|\ra$ if $p$ is even. In addition, if both $B_{1,3}$ and $B_{2,3}$ are finite, we can further obtain the rate of convergence through \eqref{thm4:eq2},
\begin{align*}
\e \la (R_{1,1}^p-\la R_{1,1}^p\ra)^2\ra&\leq \frac{36(B_{1,3}\beta_{1,p}+B_{2,3}\beta_{2,p})}{N^{(p-1)/2}}\sqrt{1-t_p}+\frac{4(\beta_{1,p}+\beta_{2,p})}{\beta_{1,p}\beta_{2,p}\sqrt{N(1-t_p)}}.
\end{align*}
\end{example}

\begin{example}
[Quenched disorder chaos in the EA model] \rm Recall the random variables $(y_{(i,j)})_{(i,j)\in e(A)}$ and $(y_i)_{i\in v(A)}$ in the EA model from Example~2. Consider independent random variables $(y_{1,(i,j)})_{(i,j)\in e(A)}$, $(y_{2,(i,j)})_{(i,j)\in e(A)}$, $(y_{1,i})_{i\in v(A)}$ and $(y_{2,i})_{i\in v(A)}$ with mean zero and variance one. For $\beta_1,\beta_2,h_1,h_2\geq 0$ and $0\leq t_b,t_s\leq 1,$ we consider two Hamiltonians of the EA model,
\begin{align*}
X_1(\rho)&=\beta_{1}\sum_{(i,j)\in e(A)}(\sqrt{t_b}y_{(i,j)}+\sqrt{1-t_b}y_{1,(i,j)})\rho_i\rho_j+h_1\sum_{i\in v(A)}(\sqrt{t_s}y_i+\sqrt{1-t_s}y_{1,i})\rho_i,\\
X_2(\tau)&=\beta_{2}\sum_{(i,j)\in e(A)}(\sqrt{t_b}y_{(i,j)}+\sqrt{1-t_b}y_{2,(i,j)})\tau_i\tau_j+h_2\sum_{i\in v(A)}(\sqrt{t_s}y_i+\sqrt{1-t_s}y_{2,i})\tau_i.
\end{align*}
Suppose that $\beta_1,\beta_2>0$ and $0\leq t_e<1.$ For any $\varepsilon>0,$ using \eqref{thm4:eq1} with $K=\varepsilon (1-t_e)^{-1},$ then
\begin{align*}
\e \la(Q_{1,1}^b-\la Q_{1,1}^b\ra)^2 \ra
&\leq \frac{16}{|E|}\sum_{e\in E}\e[|y_{1,e}|^2;|y_{1,e}|\geq \varepsilon (1-t_e)^{-1}]\\
&+\frac{16}{|E|}\sum_{e\in E}\e[|y_{2,e}|^2;|y_{2,e}|\geq \varepsilon (1-t_e)^{-1}]\\
&+36\varepsilon(\beta_{1}+\beta_{2})+\frac{(\beta_1+\beta_2)}{\beta_1\beta_2\sqrt{|e(A)|(1-t_b)}},
\end{align*}
where $$Q_{1,1}^b:=\frac{1}{|e(A)|}\sum_{(i,j)\in \in e(A)}\rho_i^1\tau_j^1\rho_i^1\tau_j^1$$ is the bond overlap. Therefore, if the decoupling rate is a small perturbation satisfying
\begin{align*}
\lim_{|e(A)|\rightarrow \infty}t_b&=1,\,\,\lim_{|e(A)|\rightarrow \infty}\sqrt{|e(A)|(1-t_b)}=\infty
\end{align*}
and for any $\varepsilon>0,$
\begin{align*}
\lim_{|e(A)|\rightarrow\infty}\frac{1}{|E|}\sum_{e\in E}\bigl(\e[|y_{1,e}|^2;|y_{1,e}|\geq \varepsilon (1-t_b)^{-1}]+\e[|y_{2,e}|^2;|y_{2,e}|\geq \varepsilon (1-t_b)^{-1}]\bigr)=0,
\end{align*}
then the bond overlap is self-averaged. One may argue exactly in the same way to obtain the self-averaging property of the site overlap $Q_{1,1}^s:=|v(A)|^{-1}\sum_{i\in v(A)}\rho_i^1\tau_i^1$.
\end{example}

\section{Universality of the Gibbs measure}\label{sec4}

In this section, we will establish universality for the Gibbs measure in the mixed $p$-spin model with mild moment matching assumptions. This naturally leads to universality in chaos phenomena and ultrametricity. 

\subsection{Ultrametricity}
Recall the Hamiltonian $H_y$ from \eqref{eq5}. For $k\geq 2,$ we assume that $(y_e)_{e\in E}$ in $H_y$ are independent of each other such that their first $k$ moments match those of a standard Gaussian random variable. In addition, we allow $\gamma$ in $H_y$ to depend on $e\in E$. In other words, we will be concerned with the generalized version of \eqref{eq5},
\begin{align*}
H_y(\si)&=\sum_{e\in E}\gamma_ey_ef_e(\si),
\end{align*}
where $\sum_{e\in E}\gamma_e^2<\infty.$ We again use $G_y$ and $\la\cdot\ra_y$ to denote the Gibbs measure and the Gibbs expectation associated to this Hamiltonian. In particular, if $(y_e)_{e\in E}$ are i.i.d. standard Gaussian random variables $(g_e)_{e\in E}$, we shall denote all these by $H_g$, $G_g$ and $\la\cdot\ra_g$. The results of this subsection will be consequences of the following theorem.

\begin{theorem}\label{thm2}
Suppose that $L$ is a measurable function depending on $(\sigma^{\ell})_{1\leq \ell\leq n}$ with $\|L\|_\infty\leq 1.$ For any sequence $(K_{e})_{e\in E}$ with $K_e\geq 1$ for $e\in E,$ we have
\begin{align}\label{thm2:eq1}
|\e\la L\ra_g-\e\la L\ra_y|&\leq  I((K_e)_{e\in E}),
\end{align}
where
\begin{align*}
I((K_e)_{e\in E})&:=\frac{4nC_{k-1,2n}}{(k-2)!}\sum_{e\in E}\gamma_e^{k}\e [|y_e|^{k};|y_e|\geq K_e]+\frac{(k+1)nC_{k,2n}}{k!}\sum_{e\in E}\gamma_e^{k+1}K_e\e|y_e|^{k}.
\end{align*}
If $B_{k+1}:=\sup_{e\in E}\e|y_{e}|^{k+1}$ is finite, then
\begin{align}\label{thm2:eq2}
|\e\la L\ra_g-\e\la L\ra_y|&\leq  \frac{(k+1)nC_{k,2n}B_{k+1}}{k!}\sum_{e\in E}\gamma_e^{k+1}.
\end{align}
\end{theorem}

\begin{lemma}\label{lem1}
Let $\nu$ be a measure on $\Sigma$. Suppose that $f$ is a measurable function on $\Sigma$ with $|f|\leq 1$ and $L$ is a measurable function on $\Sigma^n$ with $\|L\|_\infty\leq 1.$ Consider the Gibbs measure $$
G(\sigma)=\frac{\exp (\gamma xf(\sigma))\nu(\sigma)}{Z},
$$
where $Z$ is the normalizing factor. Denote by $\la\cdot\ra_x$ the Gibbs expectation associated to $G.$ For any $k\geq 1,$ we have
\begin{align*}
\Bigl|\frac{d^k}{dx^k}\la L\ra_x\Bigr|&\leq \gamma^kC_{k,n},
\end{align*}
where $C_{k,n}:=2^{k(k+1)/2}n^k$.
\end{lemma}

\begin{proof}
We claim that for $k\geq 1,$
\begin{align}\label{eq3}
\frac{d^k}{dx^k}\la L\ra_x&=\gamma^k\sum_{\ell_1=1}^{2n}\sum_{\ell_2=1}^{2^2n}\cdots\sum_{\ell_k=1}^{2^kn}c_{\ell_1,\ell_2,\ldots,\ell_k}\la Lf(\sigma^{\ell_1})f(\sigma^{\ell_2})\cdots f(\sigma^{\ell_k})\ra_x,
\end{align}
for some constants $|c_{\ell_{1},\ell_2,\ldots,\ell_k}|=1.$ If this holds, then clearly,
\begin{align*}
\Bigl|\frac{d^k}{dx^k}\la L\ra_x\Bigr|&\leq \gamma^k\prod_{i=1}^{k}(2^in)=\gamma^k 2^{k(k+1)/2}n^k=\gamma^kC_{k,n}.
\end{align*}
To show \eqref{eq3}, we proceed by induction. If $k=1,$ then
\begin{align*}
\frac{d}{dx}\la L\ra_x&=\gamma \sum_{\ell=1}^{n}\la L(f(\sigma^\ell)-f(\sigma^{\ell+n}))\ra_x
=\gamma \sum_{\ell_1=1}^{2n}c_{\ell_1}\la Lf(\sigma^\ell)\ra_x,
\end{align*}
where $c_{\ell_1}=1$ for $1\leq \ell_1\leq n$ and $-1$ for $n+1\leq \ell_1\leq 2n.$ Assume that \eqref{eq3} holds for some $k\geq 1.$ For any $Lf(\sigma^{\ell_1})\cdots f(\sigma^{\ell_k})$, we shall regard it as a function depending on $(\si^{\ell})_{1\leq \ell\leq 2^kn}$ and we compute directly to get
\begin{align*}
\frac{d}{dx}\la Lf(\sigma^{\ell_1})\cdots f(\sigma^{\ell_k})\ra_x
&=\gamma\sum_{\ell=1}^{2^kn}\la Lf(\sigma^{\ell_1})\cdots f(\sigma^{\ell_k})(f(\sigma^{\ell})-f(\sigma^{2^kn+\ell}))\ra_x\\
&=\gamma\sum_{\ell=1}^{2^{k+1}n}d_{\ell_1,\ldots,\ell_k,\ell_{k+1}}\la Lf(\sigma^{\ell_1})\cdots f(\sigma^{\ell_k})f(\sigma^{\ell})\ra_x,
\end{align*}
where $d_{\ell_1,\ldots,\ell_{k},\ell_{k+1}}=1$ if $1\leq\ell_{k+1}\leq 2^kn$ and $=-1$ if $2^kn\leq \ell_{k+1}\leq 2^{k+1}n.$ It follows that
\begin{align*}
\frac{d^{k+1}}{dx^{k+1}}\la L\ra_x
&=\gamma^{k+1}\sum_{\ell_1=1}^{2n}\cdots\sum_{\ell_k=1}^{2^kn}c_{\ell_1,\ell_2,\ldots,\ell_k}\sum_{\ell=1}^{2^{k+1}n}d_{\ell_1,\ldots,\ell_k,\ell_{k+1}}\la Lf(\sigma^{\ell_1})\cdots f(\sigma^{\ell_k})f(\sigma^{\ell})\ra_x\\
&=\gamma^{k+1}\sum_{\ell_1=1}^{2n}\cdots\sum_{\ell_k=1}^{2^kn}\sum_{\ell=1}^{2^{k+1}n}c_{\ell_1,\ell_2,\ldots,\ell_k}d_{\ell_1,\ldots,\ell_k,\ell_{k+1}}\la Lf(\sigma^{\ell_1})\cdots f(\sigma^{\ell_k})f(\sigma^{\ell})\ra_x.
\end{align*}
Letting $c_{\ell_1,\ell_2,\ldots,\ell_{k},\ell_{k+1}}=c_{\ell_1,\ell_2,\ldots,\ell_k}d_{\ell_1,\ldots,\ell_k,\ell_{k+1}}$ implies \eqref{eq3} in the case of $k+1.$ This completes the proof of our claim.
\end{proof}

\begin{proof}[Proof of Theorem \ref{thm2}] Consider the interpolated Hamiltonian between $H_g$ and $H_y,$
$$
H_s(\sigma)=\sum_{e\in E}\gamma_e(\sqrt{s}g_e+\sqrt{1-s}y_e)f_e(\sigma).
$$
Let $\la \cdot\ra_s$ be the corresponding Gibbs average and set $\phi(s)=\e \la L\ra_s.$  A direct computation gives that
\begin{align*}
\phi'(t)
&=\sum_{e\in E}\sum_{\ell=1}^n\gamma_e\e \Bigl\la L_{e,\ell}\Bigl(\frac{g_e}{\sqrt{s}}-\frac{y_e}{\sqrt{1-s}}\Bigr)\Bigr\ra_s,
\end{align*}
where $L_{e,\ell}=2^{-1}L(f_e(\sigma^\ell)-f_e(\sigma^{\ell+n})).$ Note that $L_{e,\ell}$ is a function of $(\si^\ell)_{1\leq \ell\leq n}$ with $|L_{e,\ell}|\leq 1.$ We shall think of $L_{e,\ell}$ as a function depending on $(\si^\ell)_{1\leq \ell\leq 2n}.$ For each $e\in E$, using the Gaussian integration by parts,
\begin{align*}
\e \la L_{e,\ell}g_e\ra_s&=\sqrt{s}\gamma_e \sum_{\ell=1}^{2n}\e\la L_{e,\ell}(f_e(\sigma^\ell)-f_e(\si^{\ell+2n}))\ra_s
\end{align*}
and from the approximate integration by parts \eqref{aip:eq1} together with Lemma \ref{lem1}, we obtain
\begin{align*}
&\Bigl|\e \la L_{e,\ell}y_e\ra_s-\sqrt{1-s}\gamma_e \sum_{\ell=1}^{2n}\e\la L_{e,\ell}(f_e(\sigma^\ell)-f_e(\si^{\ell+n}))\ra_s\Bigr|\\
&\leq \frac{4\gamma_e^{k-1}(1-s)^{(k-1)/2}}{(k-2)!}C_{k-1,2{n}}\e [|y_e|^{k};|y_e|\geq K_e]\\
&\quad+\frac{(k+1)\gamma_e^k(1-s)^{k/2}}{k!}C_{k,2n}K_e\e|y_e|^{k}.
\end{align*}
Combining these together and noting that $0<s<1$, the triangle inequality yields
\begin{align}
\begin{split}\label{eq9}
\Bigl|\e \Bigl\la L_{e,\ell}\Bigl(\frac{g_e}{\sqrt{s}}-\frac{y_e}{\sqrt{1-s}}\Bigr)\Bigr\ra_s\Bigr|
&=\Bigl|\frac{1}{\sqrt{s}}\e \la L_{e,\ell}g_e\ra_s-\frac{1}{\sqrt{1-s}}\e \la L_{e,\ell}y_e\ra_s\Bigr|
\leq I_e(K_e),
\end{split}
\end{align}
where
\begin{align*}
I_e(K_e)&:=\frac{4\gamma_e^{k-1}}{(k-2)!}C_{k-1,2{n}}\e [|y_e|^{k};|y_e|\geq K_e]+\frac{(k+1)\gamma_e^k}{k!}C_{k,2n}K_e\e|y_e|^{k}.
\end{align*}
Consequently, 
\begin{align*}
|\phi'(s)|&\leq n\sum_{e\in E}\gamma_eI_e(K_e)=I((K_e)_{e\in E}),
\end{align*}
which gives \eqref{thm2:eq1}. To show \eqref{thm2:eq2}, we use \eqref{aip:eq2} and follow the same argument as above.
\end{proof}

Let us now proceed to see how universality holds in the mixed $p$-spin model using Theorem \ref{thm2}. Recall the definition of the mixed $p$-spin model from Example 1. Set $E_p=\{1,\ldots,N\}^p$ for all $p\geq 2$ and $E=\cup_{p\geq 2}E_p$. Let $\gamma_e=\beta_pN^{-(p-1)/2}$ and $f_e(\sigma)=\sigma_{i_1}\cdots\sigma_{i_p}$ if $e=(i_1,\ldots,i_p)\in E_p$. The Hamiltonian of the mixed $p$-spin model in Example 1 can be rewritten as
\begin{align*}
X(\sigma)&=H_y(\si)+h\sum_{i=1}^N\sigma_i.
\end{align*}
For simplicity, we assume that $(y_{e})_{e\in E}$ are identically distributed with common law $y$. 

\begin{theorem}\label{thm6} The following two statements hold.
\begin{enumerate}
\item Suppose that $\beta_2=\beta_3=0$ and $y$ has mean zero and variance one with $\e|y|^3<\infty$. Then we have that
\begin{align}\label{thm6:eq2}
|\e \la L\ra_g-\e\la L\ra_y|&\leq \frac{3nC_{2,2n}\e|y|^{3}}{2N^{1/2}}\sum_{p\geq 4}\beta_p^{3}.
\end{align}
\item Suppose that the first four moments of $y$ match the first four moments of a standard Gaussian random variable. Then
\begin{align}
\begin{split}
\label{thm6:eq1}
|\e \la L\ra_g-\e\la L\ra_y|&\leq 2nC_{3,2n}\e[|y|^4;|y|\geq  N^{1/4}]\sum_{p\geq 2}\beta_{p}^{4}+\frac{5nC_{4,2n}\e|y|^4}{24N^{1/4}}\sum_{p\geq 2}\beta_{p}^{5}.
\end{split}
\end{align}
\end{enumerate}
\end{theorem}

\begin{proof} We start by showing the first statement. Since $\e |y|^{3}<\infty,$ using \eqref{thm2:eq2} with $k=2$ gives
\begin{align*}
|\e \la L\ra_g-\e\la L\ra_y|&\leq \frac{3nC_{2,2n}\e|y|^{3}}{2}\sum_{p\geq 4}\sum_{e\in E_p}\gamma_e^{3}\\
&=\frac{3nC_{2,2n}\e|y|^{3}}{2}\sum_{p\geq 4}\frac{\beta_p^{3}}{N^{3(p-1)/2-p}}\\
&\leq \frac{3nC_{2,2n}\e|y|^{3}}{2N^{1/2}}\sum_{p\geq 4}\beta_p^{3}.
\end{align*}
As for the second statement, we take 
\begin{align}\label{eq2}
K_e=N^{(3p-5)/4}\geq N^{1/4}
\end{align}
for all $e\in E_p$ and $p\geq 2$ and applying \eqref{thm2:eq1} with these choices of $(K_e)$, we obtain
\begin{align*}
|\e \la L\ra_g-\e\la L\ra_y|
&\leq 2nC_{3,2n}\sum_{p\geq 2}\sum_{e\in E_p}\gamma_e^4\e[|y_e|^4;|y_e|\geq K_e]
+\frac{5nC_{4,2n}}{24}\sum_{p\geq 2}\sum_{e\in E_p}\gamma_e^{5}K_e\e|y_e|^4,
\end{align*}
where from \eqref{eq2},
\begin{align*}
\sum_{p\geq 2}\sum_{e\in E_p}\gamma_e^4\e[|y_e|^4:|y_e|\geq K_e]
&=\sum_{p\geq 2}\frac{\beta_{p}^4}{N^{p-2}}\e[|y|^4;|y|\geq N^{(3p-5)/4}]\\
&\leq\e[|y|^4;|y|\geq  N^{1/4}] \sum_{p\geq p_0}\beta_{p}^{4}
\end{align*}
and
\begin{align*}
\sum_{e\in E_p}\gamma_e^{5}K_e\e|y_e|^4&=\e|y|^4\sum_{p\geq 2}\frac{\beta_{p}^{5}}{N^{3p/4-5/4}}\leq \frac{\e|y|^4}{N^{1/4}}\sum_{p\geq 2}\beta_{p}^{5}.
\end{align*}
The last two inequalities combined give us \eqref{thm6:eq1}.
\end{proof}

\begin{theorem}\label{thm:GG} Suppose that the assumption in either the first or the second statement of Theorem \ref{thm6} holds.
If ultrametricity holds in the mixed $p$-spin model with the Gaussian disorder, that is, if
\begin{align}\label{ultra}
\lim_{N\rightarrow\infty}\e\bigl\la I(R_{1,3}+\varepsilon\leq \min(R_{1,2},R_{1,3}))\bigr\ra_g=0,\,\,\forall \varepsilon>0,
\end{align} then it also holds for the disorder $y,$ where $I$ is an indicator function and $R_{\ell,\ell'}:=N^{-1}\sum_{i=1}^N\si_i^\ell\si_i^{\ell'}$ is the overlap between $\si^\ell$ and $\si^{\ell'}.$ 
\end{theorem}

\begin{proof}
Let $$L(\si^1,\si^2,\si^3)=I(R_{1,3}+\varepsilon\leq \min(R_{1,2},R_{1,3})).$$ If the assumption of the first statement of Theorem \ref{thm6} holds, then we apply \eqref{thm6:eq2} and \eqref{ultra} to get ultrametricity \eqref{ultra} under the disorder $y.$ If now the assumption of the second statement of Theorem \ref{thm6} is valid, we use \eqref{thm6:eq1} combined with \eqref{ultra} to yield \eqref{ultra} under the disorder $y.$
\end{proof}

\begin{remark}
\rm  In \cite{Panchenko}, Panchenko showed that the Ghirlanda-Guerra identities yield ultrametricity \eqref{ultra}. These identities are known to hold, for example, in the generic mixed $p$-spin model with Gaussian disorder, that is, the linear span of $\{1\}\cup\{x^p:\beta_p\neq 0\,\,\mbox{for some $p\geq 2$}\}$ is dense in $C[-1,1]$ under the supremum norm. In this case, the assumptions in Theorem \ref{thm:GG} imply universality in ultrametricity. 
\end{remark}

\subsection{Chaos phenomena}

We now proceed to establish the universality for the coupled Gibbs measure. Recall the Hamiltonians $H_{1,y}$ and $H_{2,y}$ in \eqref{eq5}. As above, we assume that the first $k$ moments of the random variables $(y_e),$ $(y_{1,e})$ and $(y_{2,e})$ match those of a standard Gaussian random variable. In addition, the parameters $t,\gamma_1,\gamma_2$ in $H_{1,y},H_{2,y}$ are allowed to depend on $e\in E$. Consider the following generalized Hamiltonians,
\begin{align}
\begin{split}\label{eq11}
H_{1,y}(\rho)&=\sum_{e\in E}\gamma_{1,e}f_e(\rho)(\sqrt{t_e}y_e+\sqrt{1-t_e}y_{1,e}),\\
H_{2,y}(\tau)&=\sum_{e\in E}\gamma_{2,e}f_e(\tau)(\sqrt{t_e}y_e+\sqrt{1-t_e}y_{2,e}),
\end{split}
\end{align} 
where $(t_e)\subset [0,1]$, $\sum\gamma_{1,e}^2$ and $\sum\gamma_{2,e}^2$ are finite. We again use the notation $G_{1,y},G_{2,y}$ to denote the Gibbs measures associated to these two Hamiltonians and $\la\cdot\ra_y$ is the Gibbs expectation for $\prod_{\ell=1}^\infty (G_{1,y}\times G_{2,y})$. In the case that $(y_e)$, $(y_{1,e})$ and $(y_{2,e})$ are i.i.d. standard Gaussian random variables $(g_e)$, $(g_{1,e})$, $(g_{2,e})$, we shall use the notations $H_{1,g},H_{2,g}$ and $\la \cdot\ra_g.$

\begin{theorem}\label{thm5}
Let $L$ be a function depending on $(\rho^\ell,\tau^\ell)_{\ell\leq n}$ with $\|L\|_\infty\leq 1.$ For any $(K_{1,e})_{e\in E}$ and $(K_{2,e})_{e\in E}$ with $K_{1,e},K_{2,e}\geq 1$, we have
\begin{align}\label{thm5:eq1}
|\e\la L\ra_g-\e \la L\ra_y|\leq I((K_{1,e})_{e\in E},(K_{2,e})_{e\in E}),
\end{align}
where
\begin{align*}
I((K_{1,e})_{e\in E},(K_{2,e})_{e\in E})&:=\frac{8nC_{{k-1},2n}}{(k-2)!}\sum_{e\in E}\sum_{j=1,2}(\gamma_{1,e}+\gamma_{2,e})^{k}\e[|y_{j,e}|^k;|y_{j,e}|\geq K_{j,e}]\\
&\quad+\frac{2(k+1)nC_{k,2n}}{k!}\sum_{e\in E}\sum_{j=1,2}(\gamma_{1,e}+\gamma_{2,e})^{k+1}K_{j,e}\e|y_{j,e}|^k.
\end{align*}
Let $B_{k+1}$ be the supremum of $\e|y_e|^{k+1}$, $\e|y_{1,e}|^{k+1}$ and $\e|y_{2,e}|^{k+1}$ for all $e\in E$.
If $B_{k+1}$ is finite, then
\begin{align}\label{thm5:eq2}
|\e\la L\ra_g-\e \la L\ra_y|&\leq \frac{2(k+1)nC_{k,2n}B_{k+1}}{k!}\sum_{e\in E}(\gamma_{1,e}+\gamma_{2,e})^{k+1}.
\end{align}
\end{theorem}

\begin{proof}
Consider the interpolated Hamiltonian for $0\leq s\leq 1,$
\begin{align*}
H_s(\rho,\tau)&=\sqrt{s}(H_{1,g}(\rho)+H_{2,g}(\tau))+\sqrt{1-s}(H_{1,y}(\rho)+H_{2,y}(\tau)).
\end{align*}
Denote by $\la\cdot\ra_s$ its Gibbs expectation and set $\phi(s)=\e\la L\ra_s.$ Note that $\phi(0)=\e\la L\ra_y$ and $\phi(1)=\e\la L\ra_g.$ Computing directly gives
\begin{align*}
\phi'(s)&=\sum_{e\in E}\gamma_{1,e}\sqrt{t_e}\sum_{\ell\leq n}\e \Bigl\la L_{1,e,\ell}\Bigl(\frac{g_{e}}{\sqrt{s}}-\frac{y_{e}}{\sqrt{1-s}}\Bigr)\Bigl\ra_s\\
&+\sum_{e\in E}\gamma_{2,e}\sqrt{t_e}\sum_{\ell\leq n}\e \Bigl\la L_{2,e,\ell}\Bigl(\frac{g_{e}}{\sqrt{s}}-\frac{y_{e}}{\sqrt{1-s}}\Bigr)\Bigl\ra_s\\
&+\sum_{e\in E}\gamma_{1,e}\sqrt{1-t_e}\sum_{\ell\leq n}\e \Bigl\la L_{1,e,\ell}\Bigl(\frac{g_{1,e}}{\sqrt{s}}-\frac{y_{1,e}}{\sqrt{1-s}}\Bigr)\Bigl\ra_s\\
&+\sum_{e\in E}\gamma_{2,e}\sqrt{1-t_e}\sum_{\ell\leq n}\e \Bigl\la L_{2,e,\ell}\Bigl(\frac{g_{2,e}}{\sqrt{s}}-\frac{y_{2,e}}{\sqrt{1-s}}\Bigr)\Bigl\ra_s,
\end{align*}
where $L_{1,e,\ell}:=2^{-1}L(f_e(\rho^\ell)-f_e(\rho^{\ell+n}))$ and $L_{2,e,\ell}=2^{-1}L(f_e(\tau^{\ell})-f_e(\tau^{\ell+n})).$ Note that we shall regard $L_{1,e,\ell}$ and $L_{2,e,\ell}$ as functions of $(\rho^\ell,\tau^\ell)_{\ell\leq 2n}$ with $|L_{1,e,\ell}|\leq 1$ and $|L_{2,e,\ell}|\leq 1.$ Following the same derivation as \eqref{eq9}, each term in the first two summations can be controlled by
\begin{align*}
\Bigl|\e \Bigl\la L_{j,e,\ell}\Bigl(\frac{g_{e}}{\sqrt{s}}-\frac{y_{e}}{\sqrt{1-s}}\Bigr)\Bigl\ra_s\Bigr|&\leq I_{j,e}(K_{j,e}),
\end{align*}
while each term in the last two summations can be controlled through
\begin{align*}
\Bigl|\e \Bigl\la L_{j,e,\ell}\Bigl(\frac{g_{j,e}}{\sqrt{s}}-\frac{y_{j,e}}{\sqrt{1-s}}\Bigr)\Bigl\ra_s\Bigr|&\leq J_{j,e}(K_{j,e}),
\end{align*}
where for $j=1,2,$
\begin{align*}
I_{j,e}(K_{j,e})&=\frac{4t_e^{(k-1)/2}C_{{k-1},2n}}{(k-2)!}\sum_{e\in E}(\gamma_{1,e}+\gamma_{2,e})^{k-1}\e[|y_{j,e}|^k;|y_{j,e}|\geq K_{j,e}]\\
&\quad+\frac{(k+1)t_e^{k/2}C_{k,2n}}{k!}\sum_{e\in E}(\gamma_{1,e}+\gamma_{2,e})^{k}K_{j,e}\e|y_{j,e}|^k
\end{align*}
and
\begin{align*}
J_{j,e}(K_{j,e})&=\frac{4(1-t_e)^{(k-1)/2}C_{{k-1},2n}}{(k-2)!}\sum_{e\in E}\gamma_{j,e}^{k-1}\e[|y_{j,e}|^k;|y_{j,e}|\geq K_{j,e}]\\
&\quad+\frac{(k+1)(1-t_e)^{k/2}C_{k,2n}}{k!}\sum_{e\in E}\gamma_j^{k}K_{j,e}\e|y_{j,e}|^k.
\end{align*}
Aiding these inequalities together and noting that $0\leq t_e\leq 1$ imply that
\begin{align*}
|\phi'(s)|&\leq \frac{4nC_{{k-1},2n}}{(k-2)!}\sum_{e\in E}\sum_{j=1,2}(\gamma_{1,e}+\gamma_{2,e})^{k-1}\gamma_{j,e}\e[|y_{j,e}|^k;|y_{j,e}|\geq K_{j,e}]\\
&+\frac{4nC_{{k-1},2n}}{(k-2)!}\sum_{e\in E}\sum_{j=1,2}\gamma_{j,e}^k\e[|y_{j,e}|^k;|y_{j,e}|\geq K_{j,e}]\\
&+\frac{n(k+1)C_{k,2n}}{k!}\sum_{e\in E}\sum_{j=1,2}(\gamma_{1,e}+\gamma_{2,e})^{k}\gamma_{j,e}K_{j,e}\e|y_{j,e}|^k\\
&+\frac{n(k+1)C_{k,2n}}{k!}\sum_{e\in E}\sum_{j=1,2}\gamma_{j,e}^{k+1}K_{j,e}\e|y_{j,e}|^k\\
&\leq I((K_{e,1})_{e\in E},(K_{e,2})_{e\in E}).
\end{align*}
Therefore, \eqref{thm5:eq1} holds. As for \eqref{thm5:eq2}, we use \eqref{aip:eq2} and Lemma \ref{lem1} with a similar argument.

%\begin{align*}
%\Bigl|\e \Bigl\la L_{1,e}\Bigl(\frac{g_{e}}{\sqrt{s}}-\frac{y_{e}}{\sqrt{1-s}}\Bigr)\Bigl\ra_s\Bigr|&\leq %t_e^{k/2}(\gamma_{1,e}+\gamma_{2,e})^kC_{k,2n}B_{k+1}\leq (\gamma_{1,e}+\gamma_{2,e})^kC_{k,2n}B_{k+1},\\
%\Bigl|\e \Bigl\la L_{2,e}\Bigl(\frac{g_{e}}{\sqrt{s}}-\frac{y_{e}}{\sqrt{1-s}}\Bigr)\Bigl\ra_s\Bigr|&\leq %t_e^{k/2}(\gamma_{1,e}+\gamma_{2,e})^kC_{k,2n}B_{k+1}\leq (\gamma_{1,e}+\gamma_{2,e})^kC_{k,2n}B_{k+1},
%\end{align*}

%while each term in the last two summations can be controlled through
%\begin{align*}
%\Bigl|\e \Bigl\la L_{1,e}\Bigl(\frac{g_{1,e}}{\sqrt{s}}-\frac{y_{1,e}}{\sqrt{1-s}}\Bigr)\Bigl\ra_s\Bigr|&\leq %(1-t_e)^{k/2}\gamma_{1,e}^kC_{k,2n}B_{1,k+1}\leq \gamma_{1,e}^kC_{k,2n}B_{1,k+1},\\
%\Bigl|\e \Bigl\la L_{2,e}\Bigl(\frac{g_{2,e}}{\sqrt{s}}-\frac{y_{1,e}}{\sqrt{1-s}}\Bigr)\Bigl\ra_s\Bigr|&\leq %(1-t_e)^{k/2}\gamma_{2,e}^kC_{k,2n}B_{2,k+1}\leq \gamma_{2,e}^kC_{k,2n}B_{2,k+1}.
%\end{align*}
%Therefore, our proof is finished since
%\begin{align*}
%|\phi'(s)|&\leq nC_{k,2n}\Bigl(B_{k+1}\sum_{e\in E}(\gamma_{1,e}+\gamma_{2,e})^{k+1}+B_{1,k+1}\sum_{e\in E}\gamma_{1,e}^{k+1}+B_{2,k+1}\sum_{e\in %E}\gamma_{2,e}^{k+1}\Bigr)\\
%&\leq 2nC_{k,2n}(B_{k+1}+B_{1,k+1}+B_{2,k+1})\sum_{e\in E}(\gamma_{1,e}+\gamma_{2,e})^{k+1}.
%\end{align*}
\end{proof}

We now explain how Theorem \ref{thm5} implies universality in chaos phenomena. Recall the two mixed $p$-spin Hamiltonians from Example 4. As in our discussion right before Theorem \ref{thm6}, we define $E=\cup_{p\geq 2}E_p$ with $E_p=\{1,\ldots,N\}^p.$ For any $e=(i_1,\ldots,i_p)\in E$, set $t_e=t_p$, $f_e(\si)=\sigma_{i_1,\ldots,i_p}$ and  
\begin{align*}
\gamma_{1,e}&=\beta_{1,p}N^{-(p-1)/2},\,\,\gamma_{2,e}=\beta_{2,p}N^{-(p-1)/2},\\
y_{e}&=y_{i_1,\ldots,i_p},\,\,y_{1,e}=y_{1,i_1,\ldots,i_p},\,\,y_{2,e}=y_{2,i_1,\ldots,i_p}.
\end{align*}
Then the two mixed $p$-spin Hamiltonians $X_1,X_2$ in Example 4 can be written as
\begin{align*}
X_1(\rho)&=H_{1,y}(\rho)+h_1\sum_{i=1}^N\rho_i,\\
X_2(\rho)&=H_{2,y}(\tau)+h_2\sum_{i=1}^N\tau_i,
\end{align*}
where $H_{y,1},H_{2,y}$ are defined in \eqref{eq11}. We shall again for simplicity consider only the situation that $(y_e)$, $(y_{1,e})$ and $(y_{2,e})$ are identically distributed with common law $y.$ Following identically the same arguments as Theorem \ref{thm6} using Theorem \ref{thm5}, we obtain an analogue of Theorem \ref{thm6} for the coupled system:

\begin{theorem}\label{thm7}
We have the following two statements.
\begin{enumerate}
\item Suppose that $\beta_{1,2}=\beta_{1,3}=\beta_{2,1}=\beta_{2,3}=0$ and $y$ has mean zero and variance one with $\e|y|^3<\infty$. Then we have that
\begin{align}\label{thm7:eq2}
|\e \la L\ra_g-\e\la L\ra_y|&\leq \frac{3nC_{2,2n}\e|y|^{3}}{N^{1/2}}\sum_{p\geq 4}(\beta_{1,p}+\beta_{2,p})^{3}.
\end{align}
\item Suppose that the first four moments of $y$ match the first four moments of a standard Gaussian random variable. Then
\begin{align}
\begin{split}
\label{thm7:eq1}
|\e \la L\ra_g-\e\la L\ra_y|&\leq 4nC_{3,2n}\e[|y|^4;|y|\geq  N^{1/4}]\sum_{p\geq 2}(\beta_{1,p}+\beta_{2,p})^{4}\\
&+\frac{5nC_{4,2n}\e|y|^4}{12N^{1/4}}\sum_{p\geq 2}(\beta_{1,p}+\beta_{2,p})^{5}.
\end{split}
\end{align}
\end{enumerate}
\end{theorem}

The phenomena of chaos are concerned about the correlation between the spin configurations sampled from systems corresponding to different external parameters such as the temperature, external field and disorder. In the mixed $p$-spin models with Gaussian disorder, the precise mathematical statement one seeks to show is that there exists some constant $q\in [-1,1]$ such that
\begin{align}\label{chaos:eq1}
\lim_{N\rightarrow \infty}\e\bigl\la I(|Q_{1,1}-q|\geq \varepsilon)\bigr\ra_g=0,\,\,\forall \varepsilon>0,
\end{align}
or in the quenched sense,
\begin{align}\label{chaos:eq2}
\lim_{N\rightarrow\infty}\e\bigl\la I(|Q_{1,1}-\la Q_{1,1}\ra_g|\geq \varepsilon)\bigr\ra_g=0,\,\,\forall \varepsilon>0,
\end{align}
where $Q_{1,1}:=N^{-1}\sum_{i=1}^N\rho_i^1\tau_i^1$ is the overlap between two systems. Note that the first statement says the overlap is essentially concentrated around a constant, while in the second statement, the overlap is concentrated around its Gibbs expectation, which depends on the environment. There are two basic types of chaos phenomena that are of great interest. For two models that share the same temperature and external field, if \eqref{chaos:eq1} or \eqref{chaos:eq2} holds due to the decoupling of the disorder, that is, $0\leq t_p<1$ for at least one $p$, then we say the model is chaotic in disorder. Likewise, assuming the two models share the same disorder and external field but they use different temperatures $(\beta_{1,p})\neq (\beta_{2,p})$, if \eqref{chaos:eq1} or \eqref{chaos:eq2} is valid, we say the model exhibits chaos in temperature. 

Rigorous results on chaos in the mixed $p$-spin model with Gaussian disorder are summarized as follows. Chaos in disorder was proved by Chatterjee \cite{Chatterjee09}, where he considered the mixed even $p$-spin model without external field and the coupling rates of the disorder are all the same $t_p=t<1$ for all even $p.$  The situation in the presence of the external field was later studied by Chen \cite{Chen11}. The results on quenched chaos in temperature in certain generic mixed $p$-spin model was first appeared in Chen and Panchenko \cite{Chen14}, where in the simplest case, they considered the situation that there exists some $p_0\geq 2$ such that $\beta_{1,p},\beta_{2,p} \neq 0$ for all $p\geq p_0$ and moreover, $\beta_{1,p}=\beta_{2,p}$ for all $p>p_0$ and $\beta_{1,p_0}\neq\beta_{2,p_0}$. The results in \cite{Chen12} were later improved in Chen \cite{Chen13}, where chaos in temperature was established in the strong sense \eqref{chaos:eq1}. 

Note that the right-hand sides of \eqref{thm7:eq2} and \eqref{thm7:eq1} both go to $0$ as $N$ goes to infinity. All results on chaos in temperature and in disorder from \cite{Chatterjee09,Chen11,Chen12,Chen13} are now also valid in the mixed $p$-spin model with disorder $y$ as long as the moment matching conditions and the temperature parameters satisfy the assumptions of Theorem \ref{thm7}. For example, universality in temperature chaos holds in the following situation.

\begin{theorem}
Assume that there exists some $p_0\geq 2$ such that 
\begin{align}
\begin{split}\label{tc}
&\beta_{1,p}=\beta_{2,p}=0,\,\,\forall \; 2\leq p\leq p_0-1,\\
&\beta_{1,p_0},\beta_{2,p_0}\neq 0\,\,\mbox{and}\,\,\beta_{1,p_0}\neq \beta_{2,p_0},\\
&\beta_{1,p}=\beta_{2,p}\neq 0,\,\,\forall \; p>p_0
\end{split}
\end{align}
and that $h_1=h_2.$
In other words, the two mixed $p$-spin models are the mixture of all $p\geq p_0$ and they share the same temperatures except at $p_0.$ Suppose that one of the following two statements holds:
\begin{enumerate}
\item $p_0\geq 4$ and $y$ has mean zero, variance one and $\e |y|^3<\infty$.
\item $p_0\geq 2$ and the first four moments of $y$ match the first four moments of the standard Gaussian random variable. 
\end{enumerate}
Then there exists some constant $q\in [-1,1]$ depending only on the temperature parameters such that
\begin{align}\label{tc:eq1}
\lim_{N\rightarrow\infty}\e\bigl\la I(|Q_{1,1}-q|\geq \varepsilon)\bigr\ra_y=0,\,\,\forall \varepsilon>0.
\end{align}
\end{theorem}

\begin{proof} 
Under the assumption \eqref{tc}, \eqref{tc:eq1} holds for the Gaussian disorder by \cite[Theorem 1]{Chen13}. Letting $L=I(|Q_{1,1}-q|\geq \varepsilon)$ and applying Theorem \ref{thm7} conclude the announced results.

\end{proof}
\begin{center}
{\bf \large Appendix}
\end{center}

\begin{proof}[Proof of Lemmas \ref{lem2}]
Using Taylor's theorem for $F$ for $(k-1)$-th and  $k$-th orders,
\begin{align}
\begin{split}\label{app:eq2}
yF(y)&=\sum_{n=0}^{k-2}\frac{F^{(n)}(0)}{n!}y^{n+1}+\frac{F^{(k-1)}(a(y))}{(k-1)!}y^{k}
\end{split}\\
\begin{split}\label{app:eq1}
&=\sum_{n=0}^{k-1}\frac{F^{(n)}(0)}{n!}y^{n+1}+\frac{F^{(k)}(b(y))}{k!}y^{k+1}
\end{split}
\end{align}
and using Taylor's theorem for $F'$ for $(k-1)$-th order,
\begin{align}
\label{app:eq3}
F'(y)&=\sum_{n=1}^{k-1}\frac{F^{(n)}(0)}{(n-1)!}y^{n-1}+\frac{F^{(k)}(c(y))}{(k-1)!}y^{k-1},
\end{align}
where  $a(y),b(y),c(y)$ are some functions depending only on $y.$ It follows that from \eqref{app:eq2} and \eqref{app:eq3},
\begin{align}
\begin{split}\label{app:eq4}
I_1&:=(yF(y)-F'(y))-F^{'}(0)y-\sum_{n=1}^{k-1}F^{(n)}(0)\Bigl(\frac{y^{n+1}}{n!}-\frac{y^{n-1}}{(n-1)!}\Bigr)\\
&=-F^{(k-1)}(0)\Bigl(\frac{y^{k}}{(k-1)!}-\frac{y^{k-2}}{(k-2)!}\Bigr)+\frac{F^{(k)}(a(y))}{k!}y^{k}-\frac{F^{(k)}(c(y))}{(k-1)!}y^{k-1}
\end{split}
\end{align}
and from \eqref{app:eq1} and \eqref{app:eq3},
\begin{align}
\begin{split}\label{app:eq5}
I_2&:=(yF(y)-F'(y))-F^{'}(0)y-\sum_{n=1}^{k-1}F^{(n)}(0)\Bigl(\frac{y^{n+1}}{n!}-\frac{y^{n-1}}{(n-1)!}\Bigr)\\
&=\frac{F^{(k)}(b(y))}{k!}y^{k+1}-\frac{F^{(k)}(c(y))}{(k-1)!}y^{k-1}.
\end{split}
\end{align}
Therefore,
\begin{align*}
|I_1|&\leq \frac{\|F^{(k-1)}\|_\infty}{(k-2)!}(|y|^{k}+|y|^{k-2})+\frac{\|F^{(k)}\|_\infty}{(k-1)!}(|y|^k+|y|^{k-1}),\\
|I_2|&\leq \|F^{(k)}\|_\infty\Bigl(\frac{|y|^{k+1}}{k!}+\frac{|y|^{k-1}}{(k-1)!}\Bigr).
\end{align*}
Taking expectation on $|y|\geq K\geq 1$ for the first inequality, we obtain
\begin{align}\label{app:eq6}
\e[|I_1|;|y|\geq K]&\leq \frac{4\|F^{(k-1)}\|_\infty}{(k-2)!}\e[|y|^{k};|y|\geq K]
\end{align}
and taking expectation on $|y|\leq K$ for the second inequality, 
\begin{align}
\begin{split}\label{app:eq7}
\e[|I_2|;|y|\leq K]&\leq \|F^{(k)}\|_\infty\Bigl(\frac{\e[|y|^{k+1};|y|\leq K]}{k!}+\frac{\e[|y|^{k-1};|y|\leq K]}{(k-1)!}\Bigr).
\end{split}
\end{align}
Since 
\begin{align*}
\e[|y|^{k+1};|y|\leq K]\leq K\e|y|^k,\\
\e[|y|^{k-1};|y|\leq K]\leq K|y|^{k-2}
\end{align*}
and from H\"{o}lder's inequality,
\begin{align}
\label{app:eq9}
\e|y|^{k-2}=\e|y|^2\e|y|^{k-2}\leq \e|y|^{k},
\end{align}
we obtain from \eqref{app:eq7} that
\begin{align}\label{app:eq8}
\e[|I_2|;|y|\leq K]&\leq \frac{(k+1)K}{k!}\|F^{(k)}\|_\infty\e|y|^k.
\end{align}
Consequently, since 
$$
|\e yF(y)-\e F'(y)|=\e|I_1+I_2|\leq \e|I_1|+\e|I_2|,
$$
the inequalities \eqref{app:eq6} and \eqref{app:eq8} gives \eqref{aip:eq1}. As for \eqref{aip:eq2}, it can be obtained by letting $K$ tend to infinity and using again \eqref{app:eq9} with $k$ replaced by $k+1$.
\end{proof}

\end{document}